\documentclass{amsproc}%
\usepackage{amssymb}
\usepackage{amsfonts}%
\usepackage{amsmath}%
\usepackage{graphicx}
\providecommand{\U}[1]{\protect\rule{.1in}{.1in}}
\theoremstyle{plain}

\newtheorem{corollary}{Corollary}

\newtheorem{definition}{Definition}

\newtheorem{lemma}{Lemma}

\newtheorem{problem}{Question}
\newtheorem{proposition}{Proposition}
\newtheorem{remark}{Remark}

\newtheorem{theorem}{Theorem}
\numberwithin{equation}{section}
\begin{document}
\title[Deformation of Hyperbolic Cone-Structures]{Deformation of Hyperbolic Cone-Structures: Study of the non-Colapsing case}
\author{Alexandre Paiva Barreto}
\address{Universidade Federal de S\~{a}o Carlos}
\email{alexandre@dm.ufscar.br}
\urladdr{}
\thanks{}
\thanks{This work is part of the doctoral thesis of the autor, made at Universit Paul
Sabatier-Toulouse under the supervision of Professor Michel Boileau and
supported by CAPES and FAPESP}
\date{March 10, 2011}
\dedicatory{Dedicated to my wife Cynthia.}
\begin{abstract}
This work is devoted to the study of deformations of hyperbolic cone
structures under the assumption that the lengths of the singularity remain
uniformly bounded over the deformation. Given a sequence $\left(  M_{i}%
,p_{i}\right)  $ of pointed hyperbolic cone-manifolds with topological type
$\left(  M,\Sigma\right)  $, where $M$ is a closed, orientable and irreducible
$3$-manifold and $\Sigma$ an embedded link in $M$. Assuming that the lengths
of the singularity remain uniformly bounded, we prove that either the sequence
$M_{i}$ collapses and $M$ is Seifert fibered or a $Sol$ manifold, or the
sequence $M_{i}$ does not collapse and in this case a subsequence of $\left(
M_{i},p_{i}\right)  $ converges to a complete Alexandrov space of dimension 3
endowed with a hyperbolic metric of finite volume on the complement of a
finite union of quasi-geodesics. We apply this result to a conjecture of
Thurston and to the case where $\Sigma$ is a small link in $M$.

\end{abstract}
\maketitle

\section{Introduction}

\qquad Fixed a closed, orientable and irreducible $3$-manifold $M$, this text
focus deformations of hyperbolic cone structures on $M$ which are singular
along a fixed embedded link $\Sigma=\Sigma_{1}\sqcup\ldots\sqcup\Sigma_{l}$ in
$M$. A \textit{hyperbolic cone structure} with topological type $\left(
M,\Sigma\right)  $ is a complete intrinsic metric on $M$ (see section $2$ for
the definition) such that every non-singular point (i.e. every point in
$M-\Sigma$) has a neighborhood isometric to an open set of $\mathbb{H}^{3}$,
the hyperbolic space of dimension $3$, and that every singular point (i.e.
every point in $\Sigma$) has a neighborhood isometric to an open neighborhood
of a singular point of $\mathbb{H}^{3}\left(  \alpha\right)  $, the space
obtained by identifying the sides of a dihedral of angle $\alpha\in\left(
0,2\pi\right]  $ in $\mathbb{H}^{3}$ by a rotation about the axe of the
dihedral. The angles $\alpha$ are called \textit{cone angles} and they may
vary from one connected component of $\Sigma$ to the other. By convention, the
complete structure on $M-\Sigma$ (see \cite{Koj2}) is considered as a
hyperbolic cone structure with topological type $\left(  M,\Sigma\right)  $
and cone angles equal to zero.

Unlike hyperbolic structures, which are rigid after Mostow, the hyperbolic
cone structures can be deformed (see \cite{HK2}). The difficulty to understand
these deformations lies in the possibility of degenerating the structure. In
other words, the Hausdorff-Gromov limit of the deformation (see section 2 for
the definition) is only an Alexandrov space which may have dimension strictly
smaller than $3$, although its curvature remains bounded from below by $-1$
(cf. \cite{Koj}). In fact, the works of Kojima, Hodgson-Kerckhoff and Fuji
(see \cite{Koj}, \cite{HK} et \cite{Fuj}) show that the degeneration of the
hyperbolic cone structures occurs if and only if the singular link of these
structures intersects itself over the deformation.

A natural way to study deformations of hyperbolic cone structures on $\left(
M,\Sigma\right)  $ is to consider sequences of hyperbolic cone structures with
topological type $\left(  M,\Sigma\right)  $ converging (in the
Hausdorff-Gromov sense) to the limit Alexandrov space. To study this kind of
sequences, we need the important notion of collapse.

\begin{definition}
We say that a sequence $M_{i}$ of hyperbolic cone-manifolds with topological
type $\left(  M,\Sigma\right)  $ collapses if, for every sequence of points
$p_{i}\in M-\Sigma$, the sequence $r_{inj}^{M_{i}-\Sigma}\left(  p_{i}\right)
$ consisting of their riemannian injectivity radii in $M_{i}-\Sigma$ converges
to zero. Otherwise, we say that the sequence $M_{i}$ does not collapse.
\end{definition}

This dichotomy is very natural and illustrates the intuitive fact that the
volume of the sequence may or may not go to zero.

We are insterested in studying the following question was raised by W.Thuston
in 80's:

\begin{problem}
\label{conjectura thurston}Let $M$ be a closed and orientable hyperbolic
manifold and suppose the existence of a simple closed geodesic $\Sigma$ in
$M$. Can the hyperbolic structure of $M$ be deformed to the complete
hyperbolic structure on $M-\Sigma$ through a path $M_{\alpha}$ of hyperbolic
cone structures with topological type $\left(  M,\Sigma\right)  $ and
parametrized by the cone angles $\alpha\in\left[  0,2\pi\right]  $?
\end{problem}

We started studying this question in \cite{Bar2}. In that paper we obtained
the following result :

\begin{theorem}
\label{teoprincipal art1}Let $M$ be a closed, orientable and irreducible
$3$-manifold and let $\Sigma=\Sigma_{1}\sqcup\ldots\sqcup\Sigma_{l}$ be an
embedded link in $M$. Suppose the existence of a sequence $M_{i}$ of
hyperbolic cone-manifolds with topological type $\left(  M,\Sigma\right)  $
and having cone angles $\alpha_{ij}\in\left(  0,2\pi\right]  $. Denote by
$\mathcal{L}_{M_{i}}\left(  \Sigma_{j}\right)  $ the length of the connected
component $\Sigma_{j}$ of $\Sigma$ in the hyperbolic cone-manifold $M_{i}$. If%
\begin{equation}
\sup\left\{  \mathcal{L}_{M_{i}}\left(  \Sigma_{j}\right)  \;;\;i\in
\mathbb{N}\text{ and }j\in\left\{  1,\ldots,l\right\}  \right\}
<\infty\label{controledocomprimento}%
\end{equation}
and the sequence $M_{i}$ collapses, then $M$ is Seifert fibered or a $Sol$ manifold.
\end{theorem}

As a consequence of this theorem, we obtained the following result yielding
some information on Thurston's question (\ref{conjectura thurston}).

\begin{corollary}
\label{corolario do art1}Let $M$ be a closed and orientable hyperbolic
$3$-manifold and suppose the existence of a finite union of simple closed
geodesics $\Sigma$ in $M$. Let $M_{\alpha}$ be a (angle decreasing)
deformation of this structure along a continuous path of hyperbolic
cone-structures with topological type $\left(  M,\Sigma\right)  $ and having
cone angles $\alpha\in\left(  L,2\pi\right]  \subset\left[  0,2\pi\right]  $
(the same for all components of $\Sigma$). If%
\begin{equation}
\sup\left\{  \mathcal{L}_{M_{\alpha}}\left(  \Sigma_{j}\right)  \;;\;\alpha
\in\left(  L,2\pi\right]  \text{ and }j\in\left\{  1,\ldots,l\right\}
\right\}  <\infty\text{,}\label{controlecomprimento2}%
\end{equation}
then every convergent (in the Hausdorff-Gromov sense) sequence $M_{\alpha_{i}%
}$, with $\alpha_{i}$ converging to $L$, does not collapses.
\end{corollary}

The additional hypothesis on the length of the singularity is automatically
verified when the holonomy representations of the hyperbolic cone structures
are convergent. This happens (cf. \cite{CS}), for example, when $\Sigma$ is a
small link in $M$ (see section 4.1 for the definition). If the deformation
proposed by Thurston in (\ref{conjectura thurston}) exists, it is a
consequence of Thurston's hyperbolic Dehn surgery Theorem that the length of
the singular link must converge to zero. In particular, we have that its
length remains uniformly bounded over the deformation. This remark makes clear
that the additional hypothesis on the length of the singularity is, in fact, a
necessary condition for the existence of the desired deformation.

When a convergent sequence of hyperbolic cone-manifolds collapses, most of the
geometric information can be lost. This happens because the dimension of the
limit Alexandrov space may be strictly smaller than 3 (see \cite{Bar2}). On
the non-collapsing case, however, the limit Alexandrov space must have
dimension 3. Our goal is to use all geometric information on the three
dimensional limit to study deformations of hyperbolic cone-structures that do
not collapse.

The principal result of this paper is the following one:

\begin{theorem}
\label{teo principal}\label{teo principal introducao}Let $M$ be a closed,
orientable and irreducible $3$-manifold and let $\Sigma=\Sigma_{1}\sqcup
\ldots\sqcup\Sigma_{l}$ be an embedded link in $M$. Suppose the existence of a
sequence $M_{i}$ of hyperbolic cone-manifolds with topological type $\left(
M,\Sigma\right)  $ and having cone angles $\alpha_{ij}\in\left(
0,2\pi\right]  $. Denote by $\mathcal{L}_{M_{i}}\left(  \Sigma_{j}\right)  $
the length of the connected component $\Sigma_{j}$ of $\Sigma$ in the
hyperbolic cone-manifold $M_{i}$. If%
\[
\sup\left\{  \mathcal{L}_{M_{i}}\left(  \Sigma_{j}\right)  \;;\;i\in
\mathbb{N}\text{ and }j\in\left\{  1,\ldots,l\right\}  \right\}  <\infty,
\]
then one of the following statements holds:

\begin{enumerate}
\item[$i.$] the sequence $M_{i}$ collapses and $M$ is Seifert fibered or a
$Sol$ manifold,

\item[$ii.$] the sequence $M_{i}$ does not collapse and there exists a
sequence of points $p_{i_{k}}\in M-\Sigma$ such that the sequence $\left(
M_{i_{k}},p_{i_{k}}\right)  $ converges in the Hausdorff-Gromov sense to a
three dimensional pointed Alexandrov space $\left(  Z,z_{0}\right)  $. The
Alexandrov space $Z$ is endowed with a (noncomplete) hyperbolic metric of
finite volume on the complement of a finite union $\Sigma_{Z}$ of
quasi-geodesics. Moreover, $Z$ is homeomorphic to $M$ (in particular, $Z$ is
compact) if there exists $\varepsilon\in\left(  0,2\pi\right)  $ such that the
cone angles $\alpha_{ij}$ belong to $\left(  \varepsilon,2\pi\right]  $.
Moreover, the following three statements are equivalent:

\begin{itemize}
\item $Z$ is compact

\item $\inf\left\{  cone-angle_{M_{i_{k}}}\left(  \Sigma_{j}\right)
\;;\;k\in\mathbb{N}\ and\ \Sigma_{j}\subset\Sigma\right\}  >0$

\item $\inf\left\{  \mathcal{L}_{M_{i_{k}}}\left(  \Sigma_{j}\right)
\;;\;k\in N\right\}  >0$, for each component $\Sigma_{j}$ of $\Sigma$.
\end{itemize}
\end{enumerate}
\end{theorem}

\begin{remark}
Note that the part (i) in the statement of the previous theorem is precisely
the theorem (\ref{teoprincipal art1}).
\end{remark}

\begin{remark}
A by-product of the above theorem is that the length of a connected component
$\Sigma_{j}$ of $\Sigma$ shrinks down to zero if and only if the same arises
for its cone angles $\alpha_{ij}$ (when $i$ goes to infinity). If the cone
angles are supposed to be the same on all of the connected components of
$\Sigma$, it follows from this fact (see Corollary
\ref{criterio ang vai p zero quando sigma e no}) that the sequence of cone
angles converges to zero if and only if the following three statements hold:

\begin{enumerate}
\item[$i.$] $\sup\left\{  \mathcal{L}_{M_{i}}\left(  \Sigma\right)
\;;\;i\in\mathbb{N}\right\}  <\infty$

\item[$ii.$] $\lim\limits_{i\rightarrow\infty}diam\left(  M_{i}\right)
=\infty$

\item[$iii.$] the sequence $M_{i}$ does not collapse.
\end{enumerate}
\end{remark}

\bigskip

As an application of Theorem \ref{teo principal introducao}, we obtain the
following result related to the Thurston's question (\ref{conjectura thurston}).

\begin{corollary}
\label{corolario da conjectura}Let $M$ be a closed and orientable hyperbolic
$3$-manifold and suppose the existence of a finite union of simple closed
geodesics $\Sigma$ in $M$. Let $M_{\alpha}$ be a deformation of this structure
along a continuous path of hyperbolic conical structures with topological type
$\left(  M,\Sigma\right)  $ and having cone angles $\alpha\in\left(
\theta,2\pi\right]  \subset\left[  0,2\pi\right]  $ (the same for all
components of $\Sigma$). Then the following statements are equivalent

\begin{enumerate}
\item[$i.$] $\theta=0$ and the path $M_{\alpha}$ extends continuously to
$\left[  0,2\pi\right]  $, where $M_{0}$ denotes $M-\Sigma$ with the complete
hyperbolic m\'{e}tric

\item[$ii.$] $\lim\limits_{\alpha\rightarrow\theta}\mathcal{L}_{M_{\alpha}%
}\left(  \Sigma\right)  =\lim\limits_{\alpha\rightarrow\theta}\sum
\limits_{i=1}^{l}\mathcal{L}_{M_{\alpha}}\left(  \Sigma_{j}\right)  =0$

\item[$iii.$] There exists a sequence $\alpha_{i}\in\left(  \theta
,2\pi\right]  $ converging to $\theta$ satisfying%
\[
\sup\left\{  \mathcal{L}_{M_{\alpha}}\left(  \Sigma_{j}\right)  \;;\;\alpha
\in\left(  \theta,2\pi\right]  \text{ and }j\in\left\{  1,\ldots,l\right\}
\right\}  <\infty
\]
and such that the sequence $diam\left(  M_{\alpha_{i}}\right)  $ goes to
infinity with $i$.
\end{enumerate}
\end{corollary}

\begin{remark}
Note that the above corollary provides a necessary and sufficient condition
for the existence of the deformation proposed by Thuston. Using the notations
in the statement of Thurston's question (\ref{conjectura thurston}), we have
that%
\[
\theta=0\Longleftrightarrow\lim\limits_{\alpha\rightarrow\theta}%
\mathcal{L}_{M_{\alpha}}\left(  \Sigma\right)  =0.
\]

\end{remark}

Supposing in addition that $M$ is not Seifert fibered and that $\Sigma$ is a
small link in $M$, we have also the following theorem (see Corollaries
\ref{petit 1} and \ref{petit 2}) providing universal constants for the
hyperbolic cone structures with topological type $\left(  M,\Sigma\right)  $.

\begin{theorem}
Let $M$ be a closed, orientable, irreducible and not Seifert fibered
$3$-manifold and let $\Sigma\ $be a small link in $M$. There exists a constant
$V=V\left(  M,\Sigma\right)  >0$ and a constant $K=K\left(  M,\varepsilon
\right)  >0$, for each $\varepsilon\in\left(  0,2\pi\right)  $, such that:

\begin{enumerate}
\item[$i.$] $Vol\left(  \mathcal{M}\right)  >V$, for every hyperbolic
cone-manifold $\mathcal{M}$ with topological type $\left(  M,\Sigma\right)  $,

\item[$ii.$] $diam\left(  \mathcal{M}\right)  <K$, for every hyperbolic
cone-manifold $\mathcal{M}$ with topological type $\left(  M,\Sigma\right)  $
and having cone angles in the interval $\left(  \varepsilon,2\pi\right]  $.
\end{enumerate}
\end{theorem}

\section{Metric Geometry}

\qquad Given a metric space $Z$, the metric on $Z$ will always be denoted by
$d_{Z}\left(  \cdot,\cdot\right)  $. The open ball of radius $r>0$ about a
subset $A$ of $Z$ is going to be denoted by%
\[
B_{Z}\left(  A,r\right)  =%
{\textstyle\bigcup\limits_{a\in A}}
\left\{  z\in Z\;;\;d_{Z}\left(  z,a\right)  <r\right\}  \text{.}%
\]
A metric space $Z$ is called a \textit{length space} (and its metric is called
intrinsic) when the distance between every pair of points in $Z$ is given by
the infimum of the lengths of all rectificable curves connecting them. When a
minimizing geodesic between every pair of points exists, we say that $Z$ is
\textit{complete}.

For all $k\in\mathbb{R}$, denote $\mathbb{M}_{k}^{2}$ the complete and simply
connected two dimensional riemannian manifold of constant sectional curvature
equal to $k$. Given a triple of points $\left(  x\;;\;y,z\right)  $ of $Z$, a
\textit{comparison triangle} for the triple is nothing but a geodesic triangle
$\Delta_{k}\left(  \overline{x},\overline{y},\overline{z}\right)  $ in
$\mathbb{M}_{k}^{2}$ with vertices $\overline{x}$, $\overline{y} $ and
$\overline{z}$ such that%
\[
d_{\mathbb{M}_{k}^{2}}\left(  \overline{x},\overline{y}\right)  =d_{Z}\left(
x,y\right)  \;\text{,}\;\;d_{\mathbb{M}_{k}^{2}}\left(  \overline{y}%
,\overline{z}\right)  =d_{Z}\left(  y,z\right)  \;\;\text{and}%
\;\;d_{\mathbb{M}_{k}^{2}}\left(  \overline{z},\overline{x}\right)
=d_{Z}\left(  z,x\right)  \text{.}%
\]
Note that a comparison triangle always exists when $k\leq0$. The
$k$\textit{-angle} of the triple $\left(  x\;;\;y,z\right)  $ is, by
definition, the angle $\measuredangle_{k}\left(  x\;;\;y,z\right)  $ of a
comparision triangle $\Delta_{k}\left(  \overline{x},\overline{y},\overline
{z}\right)  $ at the vertex $\overline{x}$ (assuming the triangle exists).

\begin{definition}
A finite dimensional (in the Hausdorff sense) length space $Z$ is called an
Alexandrov space of curvature not smaller than $k\in\mathbb{R}$ if every point
has a neighborhood $U$ such that, for all points $x,y,z\in U$, the angles
$\measuredangle_{k}\left(  x\;;\;y,z\right)  $, $\measuredangle_{k}\left(
y\;;\;x,z\right)  $ and $\measuredangle_{k}\left(  z\;;\;x,y\right)  $ are
well defined and satisfy%
\[
\measuredangle_{k}\left(  x\;;\;y,z\right)  +\measuredangle_{k}\left(
y\;;\;x,z\right)  +\measuredangle_{k}\left(  z\;;\;x,y\right)  \leq
2\pi\text{.}%
\]

\end{definition}

We point out that every hyperbolic cone-manifold is an Alexandrov space of
curvature not smaller than $-1$.

Suppose from now on that $Z$ is a $n$ dimensional Alexandrov space of
curvature not smaller than $k\in\mathbb{R}$. Consider $z\in Z$ and $\lambda
\in\left(  0,\pi\right)  $. The point $z$ is said to be $\lambda
$\textit{-strained} if there exists a set $\left\{  \left(  a_{i}%
,b_{i}\right)  \in Z\times Z\;;\;i\in\left\{  1,\ldots,n\right\}  \right\}  $,
called a $\lambda$-strainer at $z$, such that $\measuredangle_{k}\left(
x\;;\;a_{i},b_{i}\right)  >\pi-\lambda$ and%
\[
\max\left\{  \left\vert \measuredangle_{k}\left(  x\;;\;a_{i},a_{j}\right)
-\frac{\pi}{2}\right\vert ,\;\left\vert \measuredangle_{k}\left(
x\;;\;b_{i},b_{j}\right)  -\frac{\pi}{2}\right\vert ,\left\vert \measuredangle
_{k}\left(  x\;;\;a_{i},b_{j}\right)  -\frac{\pi}{2}\right\vert \right\}
<\lambda
\]
for all $i\neq j\in\left\{  1,\ldots,n\right\}  $. The set $R_{\delta}\left(
Z\right)  $ of $\lambda$-strained points of $Z$ is called the \textit{set of
}$\lambda$\textit{-regular points of }$Z$. It is a remarkable fact that
$R_{\delta}\left(  Z\right)  $ is an open and dense subset of $Z$.

Recall now, the notion of (pointed) Hausdorff-Gromov convergence (see
\cite{BBI}):

\begin{definition}
Let $\left(  Z_{i},z_{i}\right)  $ be a sequence of pointed metric spaces. We
say that the sequence $\left(  Z_{i},z_{i}\right)  $ converges in the
(pointed) Hausdorff-Gromov sense to a pointed metric space $\left(
Z,z_{0}\right)  $, if the following holds: For every $r>\varepsilon>0$, there
exist $i_{0}\in\mathbb{N}$ and a sequence of (may be non continuous) maps
$f_{i}:B_{Z_{i}}\left(  z_{i},r\right)  \rightarrow Z$ ($i>i_{0}$) such that

\begin{enumerate}
\item[$i.$] $f_{i}\left(  z_{i}\right)  =z_{0}$,

\item[$ii.$] $\sup\left\{  d_{Z^{\prime}}\left(  f_{i}\left(  z_{1}\right)
,f_{i}\left(  z_{2}\right)  \right)  -d_{Z}\left(  z_{1},z_{2}\right)
\;;\;z_{1},z_{2}\in Z\right\}  <\varepsilon$,

\item[$iii.$] $B_{Z}\left(  z_{0},r-\varepsilon\right)  \subset B_{Z}\left(
f_{i}\left(  B_{Z_{i}}\left(  z_{i},r\right)  \right)  ,\varepsilon\right)  $,

\item[$iv.$] $f_{i}\left(  B_{Z_{i}}\left(  z_{i},r\right)  \right)  \subset
B_{Z}\left(  z_{0},r+\varepsilon\right)  $.
\end{enumerate}
\end{definition}

Its a fundamental fact that the class of Alexandrov spaces of curvature not
smaller than $k\in\mathbb{R}$ is pre-compact with respect to the notion of
convergence in the Hausdorff-Gromov sense. In particular, every pointed
sequence of hyperbolic cone-manifolds with constant topological type has a
subsequence converging (in the Hausdorff-Gromov sense) to a pointed Alexandrov
space which may have dimension strictly smaller than $3$, although its
curvature remains bounded from below by $-1$.

\section{Sequences of Hyperbolic cone-manifolds}

\qquad Recall that $M$ denotes a closed, orientable and irreducible
differential manifold of dimension $3$ and that $\Sigma=\Sigma_{1}\sqcup
\ldots\sqcup\Sigma_{l}$ denotes an embedded link in $M$. A sequence of
hyperbolic cone-manifolds with topological type $\left(  M,\Sigma\right)  $
will always be denoted by $M_{i}$.

Given a sequence $M_{i}$ as above, fix indices $i\in\mathbb{N}$ and
$j\in\left\{  1,\ldots,l\right\}  $. For sufficiently small radius $R>0$, the
metric neighborhood%
\[
B_{M_{i}}\left(  \Sigma_{j},R\right)  =\left\{  x\in M_{i}\;;\;d_{M_{i}%
}\left(  x,\Sigma_{j}\right)  <R\right\}
\]
of $\Sigma$ is a solid torus embedded in $M_{i}$. The supremum of the radius
$R>0$ satisfying the above property will be called \textit{normal injectivity
radius of }$\Sigma_{j}$\textit{\ in }$M_{i}$ and it is going to be denoted by
$R_{i}\left(  \Sigma_{j}\right)  $. Anagously we can define $R_{i}\left(
\Sigma\right)  $, the \textit{normal injectivity radius of }$\Sigma$. It is a
remarkable fact (see \cite{Fuj} and \cite{HK}) that the existence of a unifom
lower bound for $R_{i}\left(  \Sigma\right)  $ ensures the existence of a
sequence of points $p_{i_{k}}\in M$ such that the sequence $\left(  M_{i_{k}%
},p_{i_{k}}\right)  $ converges in the Hausdorff-Gromov sense to a pointed
hyperbolic cone-manifold $\left(  M_{\infty},p_{\infty}\right)  $ with
topological type $\left(  M,\Sigma\right)  $. Moreover, $M_{\infty}$ must be
compact provides that cone angles of $M_{i_{k}}$ are uniformly bounded from below.

Let us also enphasize that the sequence $Vol\left(  M_{i}\right)  $ consisting
of the riemannian volumes of the hyperbolic manifolds $M_{i}-\Sigma$ is always
uniformly bounded. More precisely, we have (see \cite{Dun} and \cite{Fra})%
\begin{equation}
Vol\left(  M_{i}\right)  <Vol\left(  M_{0}\right)  \text{,}%
\label{cota superior volume}%
\end{equation}
where $M_{0}$ denotes the complete hyperbolic manifold that is homeomorphic to
$M-\Sigma$.

The purpose of this section is to prove the Theorem (\ref{teo principal}). It
is divided into two parts. The first part contains some premilinary results
whereas the remaining part deal with the proof of Theorem (\ref{teo principal}).

Let us point out that, throughout the rest of the paper, the term "component"
is going to stand for "connected component"

\subsection{Preliminary results}

Let us recall some definitons and elementary results which will be important
for the proof of Theorem (\ref{teo principal}). We will begin with the
classification of two dimensional embedded torus in $M-\Sigma$ (see
\cite{Bar2}).

\begin{lemma}
\label{Classificacao de Toros}Suppose that $M-\Sigma$ is hyperbolic and let
$T$ be a two dimensional torus embedded in $M-\Sigma$. Then $T$ separates $M
$. Moreover, one and only one of the following statements holds:

\begin{enumerate}
\item[$i.$] $T$ is parallel to a component of $\Sigma$ (hence it bounds a
solid torus in $M$),

\item[$ii.$] $T$ is not parallel to a component of $\Sigma$ and it bounds a
solid torus in $M-\Sigma$,

\item[$iii.$] $T$ is not parallel to a component of $\Sigma$ and it is
contained in a ball $B$ of $M-\Sigma$. Furthermore, $T$ bounds a region in $B$
which is homeomorphic to the exterior of a knot in $S^{3}$.
\end{enumerate}
\end{lemma}

Now let us recall the geometric classification of the thin part of a
hyperbolic manifold.

\begin{definition}
Consider $\delta>0$ and let $\mathcal{M}$ be a hyperbolic manifold of
dimension $3$ (without boundary and perhaps noncomplete). Define
$\mathcal{M}_{thin}\left(  \delta\right)  $, the $\delta$-thin part of
$\mathcal{M}$, by%
\[
\mathcal{M}_{thin}\left(  \delta\right)  =\left\{  q\in M\;;\;r_{inj}%
^{M}\left(  q\right)  <\delta\;\;et\;\;\exp_{q}\;\text{is defined on }%
B_{T_{q}M}\left(  0,3\delta\right)  \right\}  \text{.}%
\]

\end{definition}

The following result concerning the thin part of hyperbolic manifolds will be
needed later.

\begin{proposition}
\label{classif da parte fina}Let $\mathcal{M}$ be a hyperbolic manifold of
dimension $3$ (without boundary and perhaps noncomplete) of finite volume. If
$\delta>0$ is small enough, then each component of $\mathcal{M}_{thin}\left(
\delta\right)  $ contains a maximal region which is isometric to one of the
following models:

\begin{enumerate}
\item[$i.$] the quotient of a metric neighborhood of a geodesic $\gamma$ in
$\mathbb{H}^{3}$ by a loxodromic element of $PSL_{2}\left(  \mathbb{C}\right)
$ leaving $\gamma$ invariant and whose translation length is not bigger than
$\delta$,

\item[$ii.$] a parabolic cusp of rank $2$.
\end{enumerate}
\end{proposition}

This proposition is a consequence of the existenceof a Margulis foliation for
the thin part of a hyperbolic manifold. A proof for this proposition is given
in \cite[theorem 5.3]{BLP} where the authors study the thin part of hyperbolic
cone-manifolds with topological type $\left(  M,\Sigma\right)  $ and whose
cone angles are not bigger than $\pi$. Note that the condition imposed on the
cone angles is used only in the description of the singular components of the
thin part. We summarize bellow their proof for the proposition above which,
indeed, dispenses with the angle condition.

Consider a hyperbolic manifold $\mathcal{M}$ and denote by $\pi:\widetilde
{\mathcal{M}}\rightarrow\mathcal{M}$ the universal cover of $\mathcal{M}$. Let
$\delta>0$ be the constant given by the Margulis lemma (see \cite[KM]{KM},
\cite{BGS} et \cite{BLP}). Then for every component $\mathcal{P}$ of
$\mathcal{M}_{thin}\left(  \delta\right)  $, the stabilizer of a component of
$\pi^{-1}\left(  \mathcal{P}\right)  \subset\widetilde{\mathcal{M}}$ is an
elementary subgroup of $PSL_{2}\left(  \mathbb{C}\right)  $ constituted
exclusively either by parabolic or by loxodromic elements. Associated to this
group we have a canonical foliation of $\mathbb{H}^{3}$. The pull-back of this
foliation by a developing map gives a foliation on $\pi^{-1}\left(
\mathcal{P}\right)  $ which is equivariant by the action of $\pi
_{1}\mathcal{M}$. The quotient of this foliation is the Margulis foliation on
$\mathcal{P} $.

To finish the proof, it is sufficient to show that the leaves of this
foliation are two-dimensional torus. Because they are flat, it suffices
(Gauss-Bonnet) to verify that the leaves are complete. This however, is a
consequence of the fact that injectivity radius is constant on the leaves (see
\cite{Bar1}).

\subsection{Proof of the Theorem \ref{teo principal introducao}}

The purpose of this section is to study a non-collapsing sequence $M_{i}$.
Without loss of generality, this hypothesis implies the existence of a
sequence $p_{i}\in M-\Sigma$ satisfying%
\[
r_{0}=\inf\left\{  r_{inj}^{M_{i}}\left(  p_{i}\right)  \;;\;i\in
\mathbb{N}\right\}  >0\text{ ,}%
\]
and such that the sequence $\left(  M_{i},p_{i}\right)  $ converges in the
Hausdorff-Gromov sense to a pointed Alexandrov space $\left(  Z,z_{0}\right)
$. By definition of the Hausdorff-Gromov convergence, the ball $B_{Z}\left(
z_{0},r_{0}\right)  $ is isometric to a ball of radius $r_{0}$ in
$\mathbb{H}^{3}$ and this implies that $Z$ has dimension equal to $3$.

We are interested in the case where the length of the singularity remains
uniformly bounded, i.e. where%

\[
\sup\left\{  \mathcal{L}_{M_{i}}\left(  \Sigma_{j}\right)  \;;\;i\in
\mathbb{N}\;\text{,}\;j\in\left\{  1,\ldots,l\right\}  \right\}  <\infty\text{
.}%
\]
Since $Z$ has dimension 3, this assumption implies by an Ascoli-type argument
(passing to a subsequence if necessary) that each component $\Sigma_{j}$ of
$\Sigma$ satisfies one, and only one, of the following statements:

\begin{enumerate}
\item sup$\left\{  d_{M_{i}}\left(  p_{i},\Sigma_{j}\right)  \;;\;i\in
\mathbb{N}\right\}  <\infty$ and $\Sigma_{j}$ converges in the
Hausdorff-Gromov sense to a quasi-geodesic $\Sigma_{j}^{Z}\subset Z$,

\item $\lim\limits_{i\in\mathbb{N}}d_{M_{i}}\left(  p_{i},\Sigma_{j}\right)
=\infty$.
\end{enumerate}

\noindent This dichotomy allows us to write $\Sigma=\Sigma_{0}\sqcup
\Sigma_{\infty}$, where $\Sigma_{0}$ contains the components $\Sigma_{j}$ of
$\Sigma$ which satisfy Item $(1)$ and $\Sigma_{\infty}$ those that satisfy
Item $\left(  2\right)  $.

The following lemma shows that the hypothesis of non-collapsing imposes
restrictions on the length and on the cone angles of the singular components
of $\Sigma$ contained in $\Sigma_{0}$.

\begin{lemma}
\label{lim dim 3 implica controle de comp e ang}Suppose that the sequence
$M_{i}$ does not collapse and let $p_{i}\in M-\Sigma$ be a sequence of points
such that $r_{0}=\inf\left\{  r_{inj}^{M_{i}}\left(  p_{i}\right)
\;;\;i\in\mathbb{N}\right\}  >0$. If%
\[
L=\sup\left\{  \mathcal{L}_{M_{i}}\left(  \Sigma_{j}\right)  \;;\;i\in
\mathbb{N}\;\text{,}\;j\in\left\{  1,\ldots,l\right\}  \right\}  <\infty\text{
,}%
\]
then the following inequalities holds:

\begin{enumerate}
\item[$i.$] $\inf\left\{  \mathcal{L}_{M_{i}}\left(  \Sigma_{j}\right)
\;;\;i\in\mathbb{N}\;\text{,}\;\Sigma_{j}\subset\Sigma_{0}\right\}  >0$,

\item[$ii.$] $\inf\left\{  \alpha_{ij}\;;\;i\in\mathbb{N}\;\text{,}%
\;\Sigma_{j}\subset\Sigma_{0}\right\}  >0$,

\item[$iii.$] $\sup\left\{  R_{i}\left(  \Sigma_{j}\right)  \;;\;i\in
\mathbb{N}\;and\;\Sigma_{j}\subset\Sigma_{0}\right\}  <\infty$.
\end{enumerate}
\end{lemma}

\begin{proof}
Consider $\mathcal{R}>\sup\left\{  d_{M_{i}}\left(  p_{i},\Sigma_{j}\right)
\;;\;i\in\mathbb{N}\;\text{,}\;\Sigma_{j}\subset\Sigma_{0}\right\}  +r_{0} $.
Note that, by construction, $\mathcal{R}<\infty$ and $B_{M_{i}}\left(
p_{i},r_{0}\right)  \subset B_{M_{i}}\left(  \Sigma_{j},\mathcal{R}\right)  $,
for all $i\in\mathbb{N}\;$and$\;$all component $\Sigma_{j}$ of $\Sigma_{0} $.

Fix $i\in\mathbb{N}\;$and$\;$fix a component $\Sigma_{j}$ of $\Sigma_{0}$. Let
$\mathcal{A}$ be a region of $\mathbb{H}^{3}\left(  \alpha_{ij}\right)  $
which is bounded by two planes orthogonal to the singular geodesic $\sigma$ of
$\mathbb{H}^{3}\left(  \alpha_{ij}\right)  $ and having distance
$\mathcal{L}_{M_{i}}\left(  \Sigma_{j}\right)  $ between them. Using a
developing map for $M_{i}-\Sigma$ and the minimizing geodesics leaving
$\Sigma_{j}$ orthogonally, the manifold $M_{i}$ can be developed in a compact
domain $K\subset\mathcal{A}$ such that $Vol\left(  K\right)  =Vol\left(
M_{i}\right)  $.

Since $B_{M_{i}}\left(  p_{i},r_{0}\right)  \subset B_{M_{i}}\left(
\Sigma_{j},\mathcal{R}\right)  $, the development of $B_{M_{i}}\left(
p_{i},r_{0}\right)  $ in $K$ is contained in $B_{\mathbb{H}^{3}\left(
\alpha_{ij}\right)  }\left(  \sigma,\mathcal{R}\right)  \cap\mathcal{A}$. If
$V_{0}$ represents the volume of a ball of radius $r_{0}$ in $\mathbb{H}^{3}$,
we have%
\[
V_{0}=Vol\left(  B_{M_{i}}\left(  p_{i},r_{0}\right)  \right)  \leq Vol\left(
B_{\mathbb{H}^{3}\left(  \alpha_{ij}\right)  }\left(  \sigma,\mathcal{R}%
\right)  \cap\mathcal{A}\right)  =\frac{\alpha_{ij}}{2}\mathcal{L}_{M_{i}%
}\left(  \Sigma_{j}\right)  \sinh^{2}\left(  \mathcal{R}\right)
\]
and therefore%
\[
\mathcal{L}_{M_{i}}\left(  \Sigma_{j}\right)  \geq\frac{V_{0}}{\pi.\sinh
^{2}\left(  \mathcal{R}\right)  }>0\qquad\text{and}\qquad\alpha_{ij}\geq
\frac{2V_{0}}{L.\sinh^{2}\left(  \mathcal{R}\right)  }>0\text{ .}%
\]

Finally, item (iii) follows from the fact that the sequence $Vol\left(
M_{i}\right)  $ is uniformly bounded from above (see
\ref{cota superior volume}).\hfill
\end{proof}

\bigskip

With the preceding notations, set%
\[
\Sigma_{Z}=%
{\textstyle\bigcup\limits_{\Sigma_{j}\subset\Sigma_{0}}}
\Sigma_{j}^{Z}\subset Z\text{ .}%
\]
We present now the main result for the non-collapsing:

\begin{theorem}
[non-collapsing]\label{wandorema de noneffodrement}Suppose that there exists a
sequence $p_{i}\in M-\Sigma$ satisfying%
\[
r_{0}=\inf\left\{  r_{inj}^{M_{i}}\left(  p_{i}\right)  \;;\;i\in
\mathbb{N}\right\}  >0
\]
and such that the sequence $\left(  M_{i},p_{i}\right)  $ converges in the
Hausdorff-Gromov sense to a pointed Alexandrov space $\left(  Z,z_{0}\right)
$ of dimension $3$. If%
\[
\sup\left\{  \mathcal{L}_{M_{i}}\left(  \Sigma_{j}\right)  \;;\;i\in
\mathbb{N}\;\text{,}\;j\in\left\{  1,\ldots,l\right\}  \right\}  <\infty\text{
,}%
\]
then the following assertions hold:

\begin{enumerate}
\item[$i.$] $Z-\Sigma_{Z}$ is a hyperbolic manifold of finite volume whose
convex and unbounded ends are finite in number and are parabolic cusps of rank
$2$,

\item[$ii.$] $Z$ is compact (and therefore homeomorphic to $M$) if and only if
$\Sigma_{\infty}=\emptyset$,

\item[$iii.$] if $Z$ is not compact, there is a bijection between the
connected components of $\Sigma_{\infty}$ and the complete ends of
$Z-\Sigma_{Z}$. In fact, each unbounded end $C_{j}$ of $Z-\Sigma_{Z}$ is the
Hausdorff-Gromov limit of metric neighborhoods (homeomorphic to solid tori)
$B_{M_{i}}\left(  \Sigma_{j},r_{i}\right)  $ of a component $\Sigma_{j} $ of
$\Sigma_{\infty}$, where $r_{i}>0$ is an increasing sequence going off to
infinity. In addition, the cone angles $\alpha_{ij}$ and the lengths of these
components converge to $0$.
\end{enumerate}
\end{theorem}

\begin{proof}
[Proof of Item (i)]According to \cite[Lemma 2]{Fuj}, every point of
$Z-\Sigma_{Z}$ is the limit of a sequence of points of $M_{i}-\Sigma$ whose
injectivity radius is uniformly bounded from below. This implies that
$Z-\Sigma_{Z}$ is a (without boundary and noncomplete) hyperbolic manifold.
Note that the unbounded ends of $Z$ are those of $Z-\Sigma_{Z}$. In view of
Proposition (\ref{classif da parte fina}) (see also \cite[Theorem 5.3]{BLP}),
to prove item (i) it is sufficient to shows the following:\bigskip

\noindent\textbf{Claim:} $Vol\left(  Z-\Sigma_{Z}\right)  <\infty$.\bigskip

\noindent\textbf{Proof of Claim :} Suppose for contradiction the statement is
false. Let $K_{\infty}$ be a compact set of $Z-\Sigma_{Z}$ whose riemannian
volume is strictly greater than $Vol(M_{comp})$. Since the convergence is
bilipschitz on compact subsets \cite[Theorem 6.20]{CHK}, there exists an index
$i_{0}\in\mathbb{N}$ and a compact subset $K_{i_{0}}$ of $M_{i_{0}}-\Sigma$
(near $K_{\infty}$) such that%
\[
Vol\left(  M_{comp}\right)  <Vol_{M_{i_{0}}}\left(  K_{i_{0}}\right)  \leq
Vol\left(  M_{i_{0}}\right)  \text{.}%
\]
This is however impossible since $Vol\left(  M_{i_{0}}\right)  <Vol\left(
M_{comp}\right)  $ (see (\ref{cota superior volume})).\hfill$\diamond$

\hfill
\end{proof}

\bigskip

\begin{proof}
[Proof of Items (ii) and (iii)]If $Z$ is compact then $\Sigma_{\infty
}=\emptyset$. Suppose now that $Z$ is not compact. By Lemma
(\ref{lim dim 3 implica controle de comp e ang}) we can choose $R>0$ such that%
\[
B_{M_{i}}\left(  \Sigma_{j},R_{i}\left(  \Sigma_{j}\right)  \right)  \subset
B_{M_{i}}\left(  p_{i},R/2\right)
\]
for all connected component $\Sigma_{j}$ of $\Sigma_{0}$ and all $i\in N$. Let
$K$ be a compact subset of $Z$ which contains the ball $B_{Z}\left(
z_{0},R\right)  $ (and hence $\Sigma_{Z}$) in its interior and satisfies%
\[
\mathcal{Z}=Z-int\left(  K\right)  =C_{1}\sqcup\ldots\sqcup C_{m}\text{ ,}%
\]
where each $C_{k}\thickapprox$ $T^{2}\times\left[  0,\infty\right)  $ is a
cuspidal end of $Z$.

Consider a sequence $C_{1i}=T^{2}\times\left[  0,t_{i}\right]  $ of compact
subsets of $C_{1}$, where $t_{i}>0$ is an unbounded and strictly increasing sequence.

Let $\varepsilon_{i}>0$ be a sequence converging to zero. Without loss of
generality, there exists (according to \cite[Theorem 6.20]{CHK}) a sequence of
$\left(  1+\varepsilon_{i}\right)  $-bilipschitz embeddings $f_{1i}%
:C_{1i}\rightarrow M_{i}-\Sigma$ onto their images. Therefore, the sequence
$B_{1i}=f_{1i}\left(  C_{11}\right)  $ converges in the bilipschitz sense to
the compact set $C_{11}$.

Consider now a sequence of holonomy representations $\zeta_{1i}%
:\mathbb{Z\times\mathbb{Z}}\rightarrow PSL_{2}\left(  \mathbb{C}\right)  $ for
the hyperbolic structures on the interior sets $B_{1i}$. According to
\cite[Theorem 6.22]{CHK}, we can assume that%
\begin{equation}
\zeta_{1i}\circ\left(  f_{1i}\right)  _{\ast}\longrightarrow\varphi_{1}\text{
,}\label{conv homo}%
\end{equation}
where $\varphi_{1}:\mathbb{Z\times\mathbb{Z}}\rightarrow PSL_{2}\left(
\mathbb{C}\right)  $ is a holonomy representation of the hyperbolic structure
in the interior of $C_{1}$ and where $\left(  f_{1i}\right)  _{\ast
}:\mathbb{Z\times\mathbb{Z\rightarrow}}\mathbb{\pi}_{1}\left(  M-\Sigma
\right)  $ is the canonical homomorphism induced by the map $f_{1i}$.

Consider the torus $T_{1i}=f_{1i}\left(  T^{2}\times\left\{  0\right\}
\right)  $ embedded in $M-\Sigma$. Since $K$ contains the ball $B_{Z}\left(
z_{0},R\right)  $, the torus $T_{1i}$ cannot be parallel to a component
$\Sigma_{j}$ of $\Sigma_{0}$. For $i$ sufficiently large, the torus $T_{1i}$
cannot be contained in a ball of $M-\Sigma$. To see this, consider a
homotopically nontrivial loop $\gamma_{1}$ on $T^{2}\times\left\{  0\right\}
\subset C_{11}$. Since $C_{1}$ is a parabolic cusp, $\varphi_{1}\left(
\gamma_{1}\right)  $ is a nontrivial parabolic element of $PSL_{2}\left(
\mathbb{C}\right)  $ and therefore the convergence (\ref{conv homo}) implies
that $\zeta_{1i}\circ\left(  f_{1i}\right)  _{\ast}\left(  \gamma_{1}\right)
$ is not trivial for $i$ very large. The same then holds for the sequence
$\left(  f_{1i}\right)  _{\ast}\left(  \gamma_{1}\right)  $.

According to Lemma (\ref{Classificacao de Toros}), we can suppose that the
torus $T_{1i}$ bounds a solid torus $W_{1i}$ in $M$ (with perhaps a singular
soul). Note that%
\begin{equation}
\lim_{i\rightarrow\infty}diam_{M_{i}}\left(  W_{1i}\right)  =\infty
\text{,}\label{diam Wi explode}%
\end{equation}
because $f_{1i}\left(  C_{1i}\right)  \subset W_{1i}$, for all $i\in\mathbb{N}
$.

We can repeat the same construction for each cusp $C_{k}$ of $\mathcal{Z}$ in
order to obtain sequences of embedded tori $T_{ki}\subset M-\Sigma$
($k\in\left\{  1,\ldots,m\right\}  $ and $i\in\mathbb{N}$), each of then
boundy solid torus $W_{ki}$ in $M-\Sigma_{0}$. Furthermore whose diameters
become infinite with $i$. This yields a sequence of $3$-manifolds with torus
boundary%
\[
\mathcal{M}_{i}=M_{i}-%
{\textstyle\bigcup\limits_{k=1}^{m}}
W_{ki}%
\]
such that $M$ can be obtained by Dehn filling on their boundary components. By
construction, the sequence $\mathcal{M}_{i}$ converges in the Hausdorff-Gromov
sense to the compact $K$ and then (by Perelman's stability theorem
\cite{Kap}), we can assume that the manifolds $\mathcal{M}_{i}$ are all
homeomorphic to $K$.

For all $i\in\mathbb{N}$ and all $k\in\left\{  1,...,m\right\}  $, fix a
homotopically nontrivial loop $\mu_{ki}$ in $T^{2}\times\left\{  0\right\}
\subset C_{k}$ satisfying:

\begin{enumerate}
\item[$\bullet$] the loop $f_{ki}\circ\mu_{ki}$ bounds a disc in $W_{ki}$,

\item[$\bullet$] if, for some index $j\in\mathbb{N}$, a loop $\mu_{kj}$
belongs to the same homotopy class of the loop $\mu_{ki}$, then $\mu_{kj}%
=\mu_{ki}$.
\end{enumerate}

The rest of the proof is going to be divided in two cases depending on whether
or not $\Sigma_{0}$ is empty.\bigskip

\noindent\textbf{1}$^{\text{\textbf{st}}}$ \textbf{case : }$\Sigma
_{0}=\emptyset$\textit{.}\bigskip

Since the link $\Sigma$ was supposed to be non empty, it follows that
$\Sigma_{\infty}\neq\emptyset$. Since the distance between $p_{i}$ and
$\Sigma_{\infty}$ becomes infinite, we can assume that $\Sigma_{\infty}$ is
contained in the complement of $\mathcal{M}_{i}$. More precisely, we can also
assume (cf. Lemma \ref{Classificacao de Toros}) that each solid torus of
$M_{i}-\mathcal{M}_{i}$ contains at most one component of $\Sigma_{\infty}$
and, in the latter case, this component corresponds to the soul of the torus
in question.

The singular set $\Sigma_{\infty}$ has a finite number of elements. Passing to
a subsequence if necessary, we obtain an one-to-one map which associates each
component $\Sigma_{j}$ of $\Sigma_{\infty}$ to a component $C_{k_{j}}$ of
$\mathcal{Z}$, that is, the component $\Sigma_{j}$ is contained in the
component $W_{k_{j}i}$ of $M_{i}-\mathcal{M}_{i}$, for all $i\in\mathbb{N}$.

Recall that every connected component $\Sigma_{j}$ of $\Sigma_{\infty}$
satisfies $\lim\limits_{i\in\mathbb{N}}d_{M_{i}}\left(  p_{i},\Sigma
_{j}\right)  =\infty$. Since the tori $T_{k_{j}i}$ remains at a finite
distance to the points $p_{i}$ and they are parallel to the components
$\Sigma_{j}$, we must have $\lim\limits_{i\rightarrow\infty}R_{i}\left(
\Sigma_{j}\right)  =\infty$.

Since $\Sigma_{0}=\emptyset$ and thanks to \cite[Theorem 1]{Fuj}, we have that
the cone angles of $\Sigma$ converge to zero and $Z$ has a complete hyperbolic
structure whose ends are associated with components of $\Sigma_{\infty}$. In
other words, the injection defined above between the components of
$\Sigma_{\infty}$ and the components of $\mathcal{Z}$ is, in deed, a
bijection.\bigskip

\noindent\textbf{2}$^{\text{\textbf{nd}}}$\textbf{\ case : }$\Sigma_{0}%
\neq\emptyset$.\bigskip

Denote by $\Lambda$ the subset of $\left\{  1,\ldots,m\right\}  $ containing
the indices that are not associated with components of $\Sigma_{\infty}$.
Denote also by $\Omega$ the subset of $\left\{  1,\ldots,m\right\}  $
containing the indices that are associated with components of $\Sigma_{\infty
}$ whose sequence of cone angles does not converge to zero.

\begin{lemma}
\label{merid nao estabiliza}There exist $i_{0}\in\mathbb{N}$ satisfying: for
each $k\in\Lambda\cup\Omega$, the homotopy classes of loops $\mu_{ki}$
($i>i_{0}$) are pairwise distinct.
\end{lemma}

\textbf{\noindent Proof of Lemma (\ref{toros contem geod}) : }

Suppose for a contradiction that the statement of the lemma does not holds.
Without loss of generality, there exists $k_{0}\in\Lambda\cup\Omega$ such that
all loops $\mu_{k_{0}i}$ ($i\in\mathbb{N}$) belongs to the same homotopy
class. By construction, this implies that the loops $\mu_{k_{0}i}$
($i\in\mathbb{N}$) are the same loop, say $\mu$.

Suppose first that $k_{0}\in\Lambda$. By construction,
\begin{equation}
\zeta_{k_{0}i}\circ\left(  f_{k_{0}i}\right)  _{\ast}\left(  \mu\right)
=\zeta_{k_{0}i}\left(  f_{k_{0}i}\circ\mu\right)  =1_{PSL_{2}\left(
\mathbb{C}\right)  }\text{ ,}\label{de merid e nula}%
\end{equation}
for all $i\in\mathbb{N}$. Because $\varphi_{k_{0}}\left(  \left[  \mu\right]
\right)  $ is a nontrivial parabolic element of $PSL_{2}\left(  \mathbb{C}%
\right)  $, we have a contradiction.

Suppose now that $k_{0}\in\Omega$. Then $k_{0}=k_{j}$, for some component
$\Sigma_{j}$ of $\Sigma_{\infty}$ whose sequence of cone angles converges to
$\alpha_{\infty j}\neq0$. Since the maps $f_{k_{0}i}$ are $\left(
1+\varepsilon_{i}\right)  $-bilipschitz embeddings (with $\varepsilon_{i}$
shrinks down to zero), the loops $f_{k_{0}i}\circ\mu$ must have bounded lengths.

As noted in the preceeding case, the sequence $R_{i}\left(  \Sigma_{j}\right)
$ of the normal injectivity radii of the component $\Sigma_{j}$ goes off to
infinity. Since $\alpha_{\infty j}\neq0$, the sequence $\mathcal{L}_{M_{i}%
}\left(  f_{k_{0}i}\circ\mu\right)  $ formed by the lengths of the loops
$f_{k_{0}i}\circ\mu$ cannot be bounded. This is a contradiction with above
paragraph.\hfill$\diamond$\bigskip

As a consequence of the above lemma, we will show that the set $\Lambda
\cup\Omega$ is empty. To do this, the following lemma will be needed:

\begin{lemma}
\label{toros contem geod}Given $k\in\Lambda$, there exists $i_{0}=i_{0}\left(
k\right)  \in N$\ such\ that the\ solid\ tori $W_{ki}$%
\ contains\ a\ simple\ closed\ geodesic $\sigma_{ki}$, for\ every $i>i_{0}$.
\end{lemma}

\textbf{\noindent Proof of Lemma (\ref{toros contem geod}) : }Fix $k\in
\Lambda$ and let%
\[
\displaystyle\delta=\frac{\inf\left\{  r_{inj}^{Z-\Sigma_{Z}}\left(  z\right)
\;;\;z\in C_{k1}\right\}  }{2}>0\text{.}%
\]
Since the map $f_{ki_{|_{C_{k1}}}}:C_{k1}\rightarrow B_{ki}$ becomes closer
and closer to isometries, there exists $i_{1}\in\mathbb{N}$ such that
\[
r_{inj}^{M_{i}}\left(  q\right)  >\delta\text{,}%
\]
for all $i>i_{1}$ and for all $q\in B_{ki}$ (in particular, for all $q\in
T_{ki}$).\bigskip

\noindent\textbf{Claim : }\textit{There is }$i_{2}\in\mathbb{N}$\textit{\ such
that, for all }$i>i_{2}$\textit{, we can find a loop }$\gamma_{ki}%
$\textit{\ in }$W_{ki}$\textit{\ which is homotopically nontrivial in the
interior }$M-\Sigma$\textit{\ and has length smaller than }$\delta$%
\textit{.}\bigskip

\noindent\textbf{Proof of Claim : }Consider the loops constituted by two
geodesic segments with same ends and equal lengths which, furthermore, are
smaller than $\frac{\delta}{2}$. Note that there loops are always
homotopically nontrivial, otherwise we would obtain, after development, two
distinct geodesic arcs with the same ends and equal lengths in $\mathbb{H}%
^{3}$, what is not possible.

The fact that $W_{ki}$ does not admit this type of loop in its interior is
equivalent to saying that all points of $W_{ki}$ have injectivity radius not
smaller than $\frac{\delta}{2}$. This is a contradiction because the sequence
$Vol\left(  M_{i}\right)  $ is uniformly bounded from above (see
\ref{cota superior volume}) and the diameter of components $W_{ki}$ becomes
infinite.\hfill$\diamond$\bigskip

Consider $i_{o}=\max\left\{  i_{1},i_{2}\right\}  $ and fix $i>i_{0}$. Let
$\gamma_{ki}\subset W_{ki}$ be a loop as above. According to \cite[Lemma
1.2.4]{Koj}, the loop $\gamma_{ki}$ is freely homotopic (in $M-\Sigma$) to a
closed geodesic $\sigma_{ki}\subset M-\Sigma$. Moreover, the length of
$\sigma_{ki}$ is smaller than $\delta$ because the length of loops is strictly
decreasing along this homotopy. Because the points of the torus $T_{ki}$ have
injectivity radius bigger than $\delta$, all the loops involved in this
homotopy must lie entirely in the interior of $W_{ki}$. In particular,
$\sigma_{ki}\subset W_{ki}$.

If $\sigma_{ki}$ is not simple, then it gives rise to a loop $\gamma
_{ki}^{\prime}$ constituted by two geodesic segments with same ends and equal
lengths which are smaller than $\frac{\delta}{4}$. This implies that the
injectivity radius of the ends of $\gamma_{ki}^{\prime}$ is smaller than
$\frac{\delta}{4}$. We can apply the same construction for the loop
$\gamma_{ki}^{\prime}$ in order to obtain a new closed geodesic $\sigma
_{ki}\subset W_{ki}$ whose length is smaller than $\frac{\delta}{4}$. Since
the injectivity radius of points of $W_{ki}$ bounded from below by
compactness, this process must end after a finite number of steps and
therefore we can suppose that $\sigma_{ki}$ is simple. This completes the
proof of Lemma \textbf{(\ref{toros contem geod}).}\hfill$\diamond$\bigskip

The following lemma shows that $\Sigma_{\infty}$ is not empty and the cone
angles of its components goes to zero. Moreover the map between the components
of $\Sigma_{\infty}$ and the components of $\mathcal{Z}$ must be a bijection.

\begin{lemma}
\label{lambda omega e vazio}The set $\Lambda\cup\Omega$ is empty.
\end{lemma}

\textbf{\noindent Proof of Lemma (\ref{lambda omega e vazio}) : }According to
the above lemma, we can suppose the existence of a simple closed geodesic
$\sigma_{ki}$ in the solid torus $W_{ki}$, for every $i\in\mathbb{N}$ and
every $k\in\Lambda$. If the manifolds $M_{i}$ are regarded as hyperbolic
cone-manifolds with topological type $\left(  M,\Sigma^{\prime}\right)  $,
where%
\[
\Sigma^{\prime}=\Sigma\cup%
{\textstyle\bigcup\limits_{k\in\Lambda}}
\sigma_{ki}%
\]
and the cone angles on the geodesics $\sigma_{ki}$ are equal to $2\pi$, it
follows from Lemma (\ref{Classificacao de Toros}) that the tori $T_{ki}$ are
parallel to the geodesics $\sigma_{ki}$. In addition, $M-\Sigma^{\prime}$
admits a complete hyperbolic structure (see \cite{Koj2}) that will be denoted
by $\mathcal{M}_{0}$.

For all $i\in\mathbb{N}$ and all $k\in\Lambda$, denote the homotopy class of
the loop $\mu_{ik}$ by $\left(  p_{ki},q_{ki}\right)  \in\mathbb{Z\times
\mathbb{Z}}\thickapprox\pi_{1}C_{k}$. Without loss of generality, the
Thurston's hyperbolic Dehn surgery (\cite[theorem 1.13]{CHK}) gives a sequence
of complete hyperbolic manifolds $\mathcal{M}\left(  p_{i1},q_{i1}%
,\ldots,p_{im},q_{im}\right)  $ diffeomorphic to $M-\Sigma$ and such that%
\begin{equation}
V_{i}:=Vol\left(  \mathcal{M}\left(  p_{1i},q_{1i},\ldots,p_{mi}%
,q_{mi}\right)  \right)  <Vol\left(  \mathcal{M}_{0}\right)  \text{,}%
\label{volume do dehn filling e menor}%
\end{equation}
where $\left(  p_{ki},q_{ki}\right)  =\infty$, for all $i\in\mathbb{N}$ and
all $k\in\left\{  1,\ldots,m\right\}  -\Lambda$.

Since, for each $k\in\Lambda$, the pairs $\left(  p_{ki},q_{ki}\right)
_{i\in\mathbb{N}}$ are pairwise distinct (since the homotopy classes of the
loops $\mu_{ik}$ are pairwise distinct), a subsequence $\mathcal{M}\left(
p_{1i_{s}},q_{1i_{s}},\ldots,p_{mi_{s}},q_{mi_{s}}\right)  $ such that%
\[
\lim_{s\rightarrow\infty}\left\Vert \left(  p_{ki_{s}},q_{ki_{s}}\right)
\right\Vert =\lim_{s\rightarrow\infty}\left(  p_{ki_{s}}\right)  ^{2}+\left(
q_{ki_{s}}\right)  ^{2}=\infty\qquad,\;\text{for every }k\in\Lambda
\]
always exists. Thurston's hyperbolic Dehn surgery then gives%
\begin{equation}
\lim_{s\rightarrow\infty}V_{i_{s}}=Vol\left(  \mathcal{M}_{comp}\right)
\text{.}\label{vol conv para volcompleta}%
\end{equation}

Recall that the Riemannian volume of a complete hyperbolic manifold with
finite volume is a topological invariant (Mostow's Theorem). Since the
manifolds $\mathcal{M}\left(  p_{i1},q_{i1},\ldots,p_{im},q_{im}\right)  $ are
diffeomorphic, the sequence $V_{i}$ must be constant. This contradicts the
statements \ref{volume do dehn filling e menor} and
\ref{vol conv para volcompleta}. Hence $M_{i}-\mathcal{M}_{i}$ cannot have
nonsingular components. Therefore, $\Sigma_{\infty}\neq\emptyset$ and the map
between the components of $\Sigma_{\infty}$ and the components of
$\mathcal{Z}$ is a bijection.\hfill
\end{proof}

\begin{corollary}
\label{crit comp}Suppose that the sequence $M_{i}$ does not collapse and
verifies%
\[
\sup\left\{  \mathcal{L}_{M_{i}}\left(  \Sigma_{j}\right)  \;;\;i\in
\mathbb{N}\;\text{,}\;j\in\left\{  1,\ldots,l\right\}  \right\}  <\infty\text{
.}%
\]
If there is $\varepsilon\in\left(  0,2\pi\right)  $ such that the cone angles
$\alpha_{ij}$ belongs to $\left(  \varepsilon,2\pi\right]  $, then there
exists a sequence of points $p_{i_{k}}\in M-\Sigma$ such that the sequence
$\left(  M_{i_{k}},p_{i_{k}}\right)  $ converges in the Hausdorff-Gromov sense
to a compact and $3$-dimensional pointed Alexandrov space $\left(
Z,z_{0}\right)  $ (in fact homeomorphic to $M$). Moreover, there exists a
finite union of quasi-geodesics such that $Z-\Sigma_{Z}$ is a noncomplete
hyperbolic manifold of finite volume.
\end{corollary}

\begin{remark}
\label{criterios de compacidade}Suppose that $\Sigma$ is not connected. If
$\left(  M_{i},p_{i}\right)  $ is a sequence as in the statement of the
Theorem (\ref{wandorema de noneffodrement}), then the inequality%
\[
\sup\left\{  diam_{M_{i}}\left(  \Sigma\right)  \;;\;i\in\mathbb{N}\right\}
<\infty
\]
is a necessary and sufficient condition to ensure that the sequence
$diam\left(  M_{i}\right)  $ remains bounded.
\end{remark}

We have also the following less immediate corollary:

\begin{corollary}
\label{criterio ang vai p zero quando sigma e no}Let $M$ be a closed,
orientable and irreducible $3$-manifold and let $\Sigma$ be an embedded link
in $M$. Assume that there exists a sequence $M_{i}$ of hyperbolic
cone-manifolds with topological type $\left(  M,\Sigma\right)  $ and having
the same cone angles $\alpha_{i}\in\left(  0,2\pi\right]  $ for all components
of $\Sigma$. Then there is a pointed subsequence $M_{i_{k}}$ converging to
$M_{0}$ if and only if the folowing three conditions hold:

\begin{enumerate}
\item[$i.$] $\sup\left\{  \mathcal{L}_{M_{i}}\left(  \Sigma\right)
\;;\;i\in\mathbb{N}\text{ }\right\}  <\infty$,

\item[$ii.$] $\sup\left\{  diam\left(  M_{i}\right)  \;;\;i\in\mathbb{N}\text{
}\right\}  =\infty$,

\item[$iii.$] the sequence $M_{i}$ does not collapse.
\end{enumerate}
\end{corollary}

\begin{proof}
By Kojima's result (see \cite{Koj}), the existence of a subsequence $M_{i_{k}%
}$ converging to $M_{0}$ is equivalent to the convergence of the cone-angles
$\alpha_{i_{k}}$ to zero.

Suppose that the sequence $\alpha_{i}$ converges to zero. Without loss of
generality, we can assume that $\alpha_{i}\in\left(  0,\pi\right]  $, for
every $i\in\mathbb{N}$. According to \cite{Koj}, there exists a continuous
path (parametrized by cone angles) of hyperbolic cone structures with
topological type $\left(  M,\Sigma\right)  $ which connects the hyperbolic
cone structure of $M_{0}$ to the complete hyperbolic structure on $M-\Sigma$.
Moreover, by uniqueness of the hyperbolic cone structures with cone angles not
bigger than $\pi$ (see \cite{Koj}), this path contains the hyperbolic cone
structures of $M_{i}$, for every $i\in\mathbb{N}$. Then for every point $p\in
M$, the sequence $\left(  M_{i},p\right)  $ converges in the Hausdorff-Gromov
sense to $\left(  M-\Sigma,p\right)  $ with the complete hyperbolic structure.
This implies the items (ii) and (iii). The item (i) is a consequence of
Thurston's hyperbolic Dehn surgery theorem which implies that the sequence
$\mathcal{L}_{M_{i}}\left(  \Sigma\right)  $ converges to zero.

Conversely, suppose now that items (i), (ii) and (iii) are true. Then there
exists a sequence of points $p_{i_{k}}\in M-\Sigma$ satisfying%
\[
\inf\left\{  r_{inj}^{M_{i}}\left(  p_{i_{k}}\right)  \;;\;k\in\mathbb{N}%
\right\}  >0
\]
and such that the sequence $\left(  M_{i_{k}},p_{i_{k}}\right)  $ converges in
the Hausdorff-Gromov sense to a noncompact and $3$-dimensional pointed
Alexandrov space $\left(  Z,z_{0}\right)  $. Corollary (\ref{crit comp}) then
shows that the sequence $\alpha_{i}$ must converge to zero.\hfill
\end{proof}

\section{Applications}

\subsection{Small links\label{no pequeno}}

An embedded link $\Sigma$ in a $3$-manifold $M$ is called small (in $M$) if it
has an open tubular neighborhood $U$ such that $M-U$ does not contain an
embedded essential surface whose boundary is empty or an a union of meridians
of $\Sigma$. An important fact due to W.Thurston and A.Hatcher (see
\cite[Lemma 3]{HT}) is that every $3$-manifold containing a small link does
not admit an embedded essential surface.

Given a $3$-manifold $M$, let $\Sigma$ be an embedded link in $M$. Suppose
there exists a sequence $M_{i}$ of hyperbolic cone-manifolds with topological
type $\left(  M,\Sigma\right)  $ and consider the sequence $\mathcal{L}%
_{M_{i}}\left(  \Sigma\right)  $ formed by the lengths of the singular set
$\Sigma$ in $M_{i}$. As a consequence of the Culler-Shalen theory (see
\cite{CS}) we have the following proposition:

\begin{proposition}
\label{criterio petit}Let $M_{i}$ be a sequence of hyperbolic cone-manifolds
with topological type $\left(  M,\Sigma\right)  $. If $\Sigma$ is a small link
in $M$, then%
\[
\sup\left\{  \mathcal{L}_{M_{i}}\left(  \Sigma_{j}\right)  \;;\;i\in
\mathbb{N}\;\text{and}\;\Sigma_{j}\text{ component of }\Sigma\right\}
<\infty\text{.}%
\]

\end{proposition}

When $\Sigma$ is a small link in $M$, Theorem (\ref{teo principal}) yields the
following corollaries:

\begin{corollary}
\label{petit 1}Suppose that $M$ is a closed, orientable, irreducible and not
Seifert fibered $3$-manifold and let $\Sigma$ be an embedded small link in
$M$. Then there exists a constant $V=V\left(  M,\Sigma\right)  >0$ such that
$Vol\left(  \mathcal{M}\right)  >V$, for every hyperbolic cone-manifold
$\mathcal{M}$ with topological type $\left(  M,\Sigma\right)  $.
\end{corollary}

\begin{proof}
First note that $M$ is not a $Sol$ manifold. In fact every $Sol$ manifold is
foliated by essential two dimensional tori and this is not possible since
$\Sigma$ is small (see \cite[Lemma 3]{HT}).

Suppose that the lower bound $V$ does not exist. Since $\Sigma$ is small in
$M$, the non-existence of $V$ implies the existence of a sequence of
hyperbolic cone-manifolds $\mathcal{M}_{i}$ with topological type $\left(
M,\Sigma\right)  $ satisfying

\begin{enumerate}
\item[$\bullet$] $\sup\left\{  \mathcal{L}_{\mathcal{M}_{i}}\left(  \Sigma
_{j}\right)  \;;\;i\in\mathbb{N}\;\text{and}\;\Sigma_{j}\text{ component of
}\Sigma\right\}  <\infty$,

\item[$\bullet$] the sequence $Vol\left(  \mathcal{M}_{i}-\Sigma\right)  $
formed by the riemannian volumes of the hyperbolic manifolds $\mathcal{M}%
_{i}-\Sigma$ shrinks down to to zero (and therefore the sequence
$\mathcal{M}_{i}$ collapses).
\end{enumerate}

\noindent According to Therem (\ref{teo principal}), $M$ must be Seifert
fibered and this contradicts our hypothesis.\hfill
\end{proof}

\begin{corollary}
\label{petit 2}Suppose that $M$ is a closed, orientable, irreducible and not
Seifert fibered $3$-manifold and let $\Sigma$ be an embedded small link in
$M$. Given $\varepsilon\in\left(  0,2\pi\right)  $, there is a constant
$K=K\left(  M,\varepsilon\right)  >0$ such that $diam\left(  \mathcal{M}%
\right)  <K$, for every hyperbolic cone-manifold $\mathcal{M}$ with
topological type $\left(  M,\Sigma\right)  $ and having cone angles belonging
to $\left(  \varepsilon,2\pi\right]  $.
\end{corollary}

\begin{proof}
As seen in the previous corollary, $M$ is not a $Sol$ manifold. Fix
$\varepsilon\in\left(  0,2\pi\right)  $ and suppose that the upper bpund $K$
does not exist. Since $\Sigma$ is small in $M$, the non-existence of $K$
implies the existence of a sequence of hyperbolic cone-manifolds
$\mathcal{M}_{i}$ with topological type $\left(  M,\Sigma\right)  $, having
cone angles $\alpha_{ji}\in\left(  \varepsilon,2\pi\right]  $ and satisfying

\begin{enumerate}
\item[$i.$] $\sup\left\{  \mathcal{L}_{\mathcal{M}_{i}}\left(  \Sigma
_{j}\right)  \;;\;i\in\mathbb{N}\;\text{and}\;\Sigma_{j}\text{ component of
}\Sigma\right\}  <\infty$,

\item[$ii.$] the sequence $diam\left(  \mathcal{M}_{i}\right)  $ formed by the
diameters of the hyperbolic cone-manifolds $\mathcal{M}_{i}$ go to infinity.
\end{enumerate}

\noindent Since $M$ is neither Seifert fibered nor a $Sol$ manifold, it
follows from item (i) and Theorem (\ref{teo principal}) that the sequence
$\mathcal{M}_{i}$ does not collapse. Moreover, since the cone angles
$\alpha_{ji}$ belong to $\left(  \varepsilon,2\pi\right]  $, it follows that
the sequence $diam\left(  \mathcal{M}_{i}\right)  $ is bounded and this yields
a contradiction with item (ii).\hfill
\end{proof}

\subsection{Proof of Corollary \ref{corolario da conjectura}}

First, we would like to recall that the existence of a deformation $M_{\alpha
}$ as in Corollary \ref{corolario da conjectura} is a consequence of the Local
Deformation Theorem due to Hodgson and Kerckhoff \cite{HK2}.

\begin{proof}
The implication ($i\Rightarrow ii$) is immediate (see \cite{Koj}). Suppose now
that the sequence $\mathcal{L}_{M_{\alpha}}\left(  \Sigma\right)  $ converges
to $0$ when $\alpha$ converges to $\theta$. Then
\[
\sup\left\{  \mathcal{L}_{M_{\alpha_{i}}}\left(  \Sigma_{j}\right)
\;;\;i\in\mathbb{N}\text{ and }\Sigma_{j}\text{ component of }\Sigma\right\}
<\infty\text{,}%
\]
for every sequence $\alpha_{i}\in\left(  \theta,2\pi\right]  $ converging to
$\theta$. Consider such a sequence $\alpha_{i}$. Since $M$ is hyperbolic (and
therefore is neither Seifert fibered nor a $Sol$ manifold), it follows from
theorem (\ref{teo principal}) that the sequence $M_{\alpha_{i}}$ does not
collapse. Moreover, since the sequence $\mathcal{L}_{M_{\alpha_{i}}}\left(
\Sigma\right)  $ converges to zero, we must have $\lim\limits_{i\rightarrow
\infty}diam\left(  M_{\alpha_{i}}\right)  =\infty$. This concludes the proof
of the implication ($ii\Rightarrow iii$).

To prove ($iii\Rightarrow i$) take a sequence $\alpha_{i}$ satisfying item
($iii$). Again by Theorem (\ref{teo principal}), it follows that the sequence
$M_{\alpha_{i}}$ does not collapse. Moreover, since the sequence $diam\left(
M_{\alpha_{i}}\right)  $ is not bounded, we must have $\theta=0$ because all
the components of $\Sigma$ have the same cone angle. Then, by Kojima's work
(see \cite{Koj}), it follows that $M_{i}$ converges (in the Hausdorff-Gromov
sense) to $M_{0}$.\hfill
\end{proof}

\end{document}